\documentclass{article}

\usepackage{arxiv}

\usepackage[utf8]{inputenc} % allow utf-8 input
\usepackage[T1]{fontenc}    % use 8-bit T1 fonts
\usepackage{hyperref}       % hyperlinks
\usepackage{url}            % simple URL typesetting
\usepackage{booktabs}       % professional-quality tables
\usepackage{amsfonts}       % blackboard math symbols
\usepackage{nicefrac}       % compact symbols for 1/2, etc.
\usepackage{microtype}      % microtypography
\usepackage{lipsum}		% Can be removed after putting your text content
\usepackage{graphicx}
\usepackage{mathtools}
\usepackage{amssymb}
\usepackage{amsthm}
\usepackage{amsmath}
\usepackage{doi}
\usepackage{mathtools, cases}

\allowdisplaybreaks

\newcommand{\sgn}{\operatorname{sgn}}

\renewcommand{\cosh}{\operatorname{cosh}}

\newtheorem{theorem}{Theorem}

\newtheorem{remark}{Remark}

\title{Pseudoprocesses related to higher-order equations of vibrations of rods}

\author{ %\href{https://orcid.org/0000-0000-0000-0000}{Manfred Marvin Marchione}\thanks{Use footnote for providing further information about author (webpage, alternativeaddress)---\emph{not} for acknowledging funding agencies.}\\
	\href{https://orcid.org/0000-0002-6163-044X}{Manfred Marvin Marchione}\\
	Department of Statistical Sciences\\
	Sapienza University of Rome\\
	\texttt{manfredmarvin.marchione@uniroma1.it} \\
	\And
	\href{https://orcid.org/0000-0002-6421-533X}{Enzo Orsingher} \\
	Department of Statistical Sciences\\
	Sapienza University of Rome\\
	\texttt{enzo.orsingher@uniroma1.it}}

\date{\today}

\renewcommand{\headeright}{}
\renewcommand{\undertitle}{}

\begin{document}
\maketitle

\begin{abstract}
In this paper we study Fresnel pseudoprocesses whose signed measure density is a solution to a higher-order extension of the equation of vibrations of rods. We also investigate space-fractional extensions of the pseudoprocesses related to the Riesz operator. The measure density is represented in terms of generalized Airy functions which include the classical Airy function as a particular case. We prove that the Fresnel pseudoprocess time-changed with an independent stable subordinator produces genuine stochastic processes. In particular, if the exponent of the subordinator is chosen in a suitable way, the time-changed pseudoprocess is identical in distribution to a mixture of stable processes. The case of a mixture of Cauchy distributions is discussed and we show that the symmetric mixture can be either unimodal or bimodal, while the probability density function of an asymmetric mixture can possibly have an inflection point.
\end{abstract}


\section{Introduction}
\noindent In mathematical physics, the vibrations of rods are described by the partial differential equation
\begin{equation}\label{rodseq}\frac{\partial^2 u}{\partial t^2}=-\kappa^2\frac{\partial ^{4}u}{\partial x^{4}},\qquad x\in\mathbb{R},\;t>0\end{equation}
\noindent where the constant $\kappa^2$ is related to the structure of the vibrating rod (see Courant and Hilbert \cite{hilbert}, pag. 244-246). A probabilistic analysis of equation (\ref{rodseq}) has been performed by Orsingher and D'Ovidio \cite{orsingherdovidio_fresnel} for $\kappa=\frac{1}{2}$ by decomposing the equation into two Schr\"{o}dinger-type equations
\begin{equation}\label{schrodinger}\left(\frac{\partial}{\partial t}-\frac{i}{2}\frac{\partial^{2}}{\partial x^{2}}\right)\left(\frac{\partial}{\partial t}+\frac{i}{2}\frac{\partial^{2}}{\partial x^{2}}\right)u=0.\end{equation}
\noindent The Schr\"{o}dinger equation was examined from a probabilistic point of view by Ibragimov et al. \cite{ibragimov} and Platonova and Tsykin \cite{platonova}. By treating the Schr\"{o}dinger-type equations as heat equations with imaginary time, the fundamental solution to (\ref{rodseq}) can be expressed in the form
\begin{align}u(x,t)=&\frac{1}{2}\left[P(B(s)\in dx)\Big|_{s=it}+P(B(s)\in dx)\Big|_{s=-it}\right]/dx\nonumber\\
=&\frac{1}{2}\left[\frac{e^{-\frac{x^2}{2it}-i\frac{\pi}{4}}}{\sqrt{2\pi t}}+\frac{e^{\frac{x^2}{2it}+i\frac{\pi}{4}}}{\sqrt{2\pi t}}\right]=\frac{1}{\sqrt{2\pi t}}\cos\left(\frac{x^2}{2t}-\frac{\pi}{4}\right)\label{fresneldensity}\end{align}
\noindent where $B$ is a standard Brownian motion. The probabilistic interpretation of formula (\ref{fresneldensity}) is that a Brownian motion is started at the origin at time $t=0$ and it moves either with increasing or decreasing imaginary time with equal probabilities. The direction of the time is chosen once forever at $t=0$ and can't be changed.\\
\noindent In the same paper, the authors constructed a signed measure $\mathbb{P}$ on the space of real-valued measurable functions $x(t),\;t\ge0$, defined by
\begin{equation}\label{signedmeasuredef}\mathbb{P}\left(\cap_{j=1}^m \{a_j\le x(t_j)\le b_j\}\right)=\int_{a_1}^{b_1}\ldots\int_{a_m}^{b_m}\prod_{j=1}^m u(x_j-x_{j-1},\;t_j-t_{j-1})\;dx_j\end{equation}
\noindent with $t_0=0$ and $x_0=0$. The fundamental solution (\ref{fresneldensity}) can be viewed as a signed density function related to a pseudoprocess $F(t)$ which was called the Fresnel pseudoprocess by the authors. 
The signed measure $\mathbb{P}$ is well defined on the algebra $\mathcal{F}$ containing all the cylinder sets and is finitely additive on $\mathcal{F}$. Since $\mathbb{P}$ is signed, Kolmogorov's extension theorem cannot be applied to extend $\mathbb{P}$ to the $\sigma$-algebra generated by $\mathcal{F}$. Moreover, $\mathbb{P}$ has infinite total variation (see Krylov \cite{krylov}) and therefore cannot be extended to $\sigma\left(\mathcal{F}\right)$.\\
\noindent The construction method described by formula (\ref{signedmeasuredef}), which resembles the classical construction of the Wiener measure, was used by different authors for studying pseudoprocesses related to higher-order heat-type equations of the form \begin{equation}\label{htype}\frac{\partial u}{\partial t}=c_n\frac{\partial ^n u}{\partial x^n}\end{equation}
\noindent for suitable coefficients $c_n$. The even-order case was studied from a probabilistic point of view by Krylov \cite{krylov}. The odd-order case was first examined by Orsingher \cite{orsingher} and later developed by Lachal \cite{lachal, lachal2} and Orsingher and D'Ovidio \cite{orsingherdovidio_prob}. The authors obtained the exact distributions of various functionals of pseudoprocesses like the maximum and the sojourn time.\\
\noindent A generalization of equation (\ref{rodseq}) was considered by De Gregorio and Orsingher \cite{degregorioorsingher} who introduced pseudoprocesses related to equations of the form 
\begin{equation}\label{horodseq}\frac{\partial^2 u}{\partial t^2}=(-1)^{n+1}\frac{\partial ^{2n}u}{\partial x^{2n}},\qquad x\in\mathbb{R},\;t>0,\;n\ge2.\end{equation}
%\noindent By expressing equation (\ref{horodseq}) in the form
%$$\left(\frac{\partial}{\partial t}-e^{-i\frac{\pi}{2}(n+1)}\frac{\partial^{n}}{\partial x^{n}}\right)\left(\frac{\partial}{\partial t}+e^{-i\frac{\pi}{2}(n+1)}\frac{\partial^{n}}{\partial x^{n}}\right)u=0$$
%\noindent it is shown that the fundamental solution to (\ref{horodseq}) can be expressed as a linear combination of the solutions to the higher-order Schr\"{o}dinger-type equations
%\begin{equation}
%\label{decomp}
%\begin{dcases}
%\frac{\partial v_{1,n}}{\partial t}=e^{-i\frac{\pi}{2}(n+1)}\frac{\partial^{n}v_{1,n}}{\partial x^{n}}\\
%\frac{\partial v_{2,n}}{\partial t}=-e^{-i\frac{\pi}{2}(n+1)}\frac{\partial^{n}v_{2,n}}{\partial x^{n}}
%\end{dcases}
%\end{equation}
%\noindent The equations in (\ref{decomp}) can be solved by a Fourier transform approach and, by assuming initial conditions $v_{1,n}(x,0)=v_{2,n}(x,0)=\delta(x)$, we have that $$v_{1,n}(x,t)=\frac{1}{2\pi}\int_{-\infty}^{+\infty}e^{-i\gamma x-(-1)^nit\gamma^{n}}d\gamma,\qquad v_{2,n}(x,t)=\frac{1}{2\pi}\int_{-\infty}^{+\infty}e^{-i\gamma x+(-1)^nit\gamma^{n}}d\gamma.$$
\noindent The authors expressed the fundamental solution $u_{2n}(x,t)$ to equation (\ref{horodseq}) in the form
\begin{equation}u_{2n}(x,t)=\frac{1}{\pi}\int_0^{+\infty}\cos\left(\gamma x\right)\cos\left(\gamma^n t\right)d\gamma.\label{fundamentalsolution}\end{equation}
\noindent and they pointed out that, for $n=3$, the solution (\ref{fundamentalsolution}) can be written as
\begin{equation}\label{airyrepr}u_{6}(x,t)=\frac{1}{2(3t)^{\frac{1}{3}}}\left[\normalfont{\text{Ai}}\left(\frac{x}{(3t)^{\frac{1}{3}}}\right)+\normalfont{\text{Ai}}\left(-\frac{x}{(3t)^{\frac{1}{3}}}\right)\right],\;\; x\in\mathbb{R},\;t>0\end{equation}
\noindent where $\normalfont{\text{Ai}}(x)$ is the Airy function of the first kind.\\
\noindent In this paper we show that the representation (\ref{airyrepr}) can be extended to all values of $n$ by using the generalized Airy functions
\begin{equation}\label{fractionalairydef}\normalfont{\text{Ai}}_{\alpha}(x)=\frac{1}{\pi}\int_0^{+\infty}\cos\left(sx+\frac{s^{\alpha}}{\alpha}\right)ds,\qquad x\in\mathbb{R},\;\alpha>1\end{equation}
\noindent The generalized Airy function (\ref{fractionalairydef}), which coincides with the classical Airy function for $\alpha=3$, was first introduced by Ansari and Askari \cite{ansari} and Askari and Ansari \cite{askari} in the integer case $\alpha=n$, $n\ge 2$, as a solution to the ordinary differential equation
$$y^{(n)}(x)+\gamma_nxy(x)=0,\qquad \gamma_n=\begin{cases}(-1)^{\frac{n}{2}}&\text{for even}\;n\\-1&\text{for odd}\;n\end{cases}$$
\noindent and it was generalized to all possible values of $\alpha>1$ by Marchione and Orsingher \cite{marchione} which obtained the function (\ref{fractionalairydef}) as the solution to a fractional diffusion equation involving an higher-order Riesz-Feller operator.\\
\noindent The fundamental solution $u_{2n}(x,t)$ to equation (\ref{horodseq}) can be expressed in the form
\begin{equation}\label{genairyrepr}u_{2n}(x,t)=\frac{1}{2(nt)^{\frac{1}{n}}}\left[\normalfont{\text{Ai}}_{n}\left(\frac{x}{(nt)^{\frac{1}{n}}}\right)+\normalfont{\text{Ai}}_{n}\left(-\frac{x}{(nt)^{\frac{1}{n}}}\right)\right].\end{equation}
\noindent However, we observe that a distinction must be made between the even case $n=2r$ and the odd case $n=2r+1$. While in the even case only symmetric solutions of the form (\ref{genairyrepr}) are admitted, for odd values of $n$ we obtain asymmetric solutions of the form
\begin{equation}\label{asymmetric_pseudodensity_frac_intro}u_{2n, p}(x,t)=\frac{1}{(n t)^{\frac{1}{n}}}\Biggl[p\cdot \normalfont{\text{Ai}}_{n}\left(-\frac{x}{(n t)^{\frac{1}{n}}}\right)+(1-p)\cdot\normalfont{\text{Ai}}_{n}\left(\frac{x}{(n t)^{\frac{1}{n}}}\right)\Biggr]\end{equation}
\noindent Moreover, the probabilistic interepretation we give to formula (\ref{asymmetric_pseudodensity_frac_intro}) is substantially different from that given by Orsingher and D'Ovidio \cite{orsingherdovidio_fresnel} which is described by formula (\ref{fresneldensity}).\\
\noindent Our analysis begins with the study of the symmetric pseudo-density (\ref{genairyrepr}) which can be expressed in the form
$$u_{2n}(x,t)=\frac{1}{n\pi t^{\frac{1}{n}}}\sum_{k=0}^{\infty}\left(\frac{x}{t^{\frac{1}{n}}}\right)^{2k}\frac{\Gamma\left(\frac{2k+1}{n}\right)}{(2k)!}\sin\left(\pi\frac{(2k+1)(n+1)}{2n}\right)$$
\noindent Moreover, we show that the pseudo-density admits the probabilistic representation
\begin{equation}\label{introprob}u_{2n}(x,t)=\frac{1}{\pi x}\mathbb{E}\left[\sin\left(a_nxG_{n}(1/t)\right)\cosh\left(b_nxG_{n}(1/t)\right)\right]\end{equation}
\noindent where $G_{\gamma}(\tau)$ is a random variable with Weibull distribution having probability density function $g_{\gamma}(y;\tau)=\gamma\frac{y^{\gamma-1}}{\tau}\exp\left(-\frac{y^{\gamma}}{\tau}\right)$ with $y,\gamma,\tau>0$, $a_n=\cos\frac{\pi}{2n}$ and $b_n=\sin\frac{\pi}{2n}.$
\noindent Thus, the function $u_{2n}(x,t)$ can be interpreted as the expected value of amplified oscillations with Weibull distributed random parameters.\\
\noindent We then introduce a fractional extension of the Fresnel pseudoprocess by studying the space-fractional partial differential equation
\begin{equation*}\frac{\partial^2 u_{2\alpha}}{\partial t^2}=\frac{\partial ^{2\alpha}u_{2\alpha}}{\partial |x|^{2\alpha}},\qquad x\in\mathbb{R},\;t>0\end{equation*}
\noindent whose solution can be expressed in the form
$$u_{2\alpha,p}(x,t)=\frac{1}{(\alpha t)^{\frac{1}{\alpha}}}\Biggl[p\cdot \normalfont{\text{Ai}}_{\alpha}\left(-\frac{x}{(\alpha t)^{\frac{1}{\alpha}}}\right)+(1-p)\cdot\normalfont{\text{Ai}}_{\alpha}\left(\frac{x}{(\alpha t)^{\frac{1}{\alpha}}}\right)\Biggr]$$
\noindent where the Airy function $\normalfont{\text{Ai}}_{\alpha}(x)$ is defined in formula (\ref{fractionalairydef}).\\
\noindent The function $u_{2\alpha,p}(x,t)$ can be interepreted as the density of a fractional Fresnel pseudoprocess $F_{2\alpha}^{(p)}(t)$. Moreover, a probabilistic representation for $u_{2\alpha,p}(x,t)$ resembling formula (\ref{introprob}) holds.\\
\noindent We finally study the fractional Fresnel pseudoprocess $F_{2\alpha}^{(p)}(t)$ time-changed with an independent Lévy stable subordinator $S_{\theta}(t)$ having exponent $\theta$. We prove that, if the exponent $\theta$ is chosen in a suitable way, the time-changed pseudoprocess
$$Y_{2\alpha,\theta}^{(p)}(t)=F_{2\alpha}^{(p)}(S_{\theta}(t))$$
\noindent becomes a genuine stochastic process. In particular, we obtain that
\begin{equation}\label{introconclusion}Y_{2\alpha,\theta}^{(p)}(t)\overset{i.d.}{=}
\begin{dcases}
H_{\nu}(\sigma, \beta, \mu; t)\qquad&\text{with prob.}\;p\\
-H_{\nu}(\sigma, \beta, \mu; t)\qquad&\text{with prob.}\;1-p
\end{dcases}\end{equation}
\noindent where $H_{\nu}(\sigma, \beta, \mu; t)$ is an asymmetric Lévy stable process having suitable parameters which depend on $\alpha$ and $\theta$.\\
\noindent A special case of the relationship (\ref{introconclusion}) is obtained for $\theta=\frac{1}{\alpha}$. For such a choice of $\theta$, the distribution of the subordinated pseudoprocess is a mixture of Cauchy distributions. The case $\alpha=2,\;p=\frac{1}{2}$ was examined by Orsingher and D'Ovidio \cite{orsingherdovidio_fresnel} who obtained the symmetric probability density function \begin{equation}\label{specialbimodal}\mathbb{P}(Y_{4,\frac{1}{2}}(t)\in dx)/dx=\frac{t}{\pi\sqrt{2}}\;\frac{x^2+t^2}{x^4+t^4}.\end{equation}
It can be shown that the probability density (\ref{specialbimodal}) is bimodal and its two modes coincide with the points $x=\pm t\sqrt{\sqrt{2}-1}$. In general, we discover that, in the symmetric case $p=\frac{1}{2}$, probability density function of the subordinated pseudoprocess $$Y_{2\alpha,\frac{1}{\alpha}}(t)=F_{2\alpha}(S_{\frac{1}{\alpha}}(t))$$ is bimodal only for $1<\alpha<3$. Outside of this interval, that is for $\alpha\ge3$, the symmetric mixture distribution is unimodal.\\
\noindent In the case of asymmetric superposition of Cauchy densities emerging in the study of the subordinated pseudoprocess $Y_{2\alpha,\frac{1}{\alpha}}^{(p)}(t)$, the structure of the distribution is considerably more complicated. In this case we show that the probability density function can possibly have an inflection point. We prove this fact by finding a particular choice of $\alpha$ and $p$ such that the inflection point emerges.

%%%%%%%%%%%%%%%%%%%%%%%%%%%%%%%%%%%%%%%%%%%%%%%%%%%
%%%%%%%%%%%%%%%%%%%%%%%%%%%%%%%%%%%%%%%%%%%%%%%%%%%
%%%%%%%%%%%%%%%%%%%%%%%%%%%%%%%%%%%%%%%%%%%%%%%%%%%
%%%%%%%%%%%%%%%%%%%%%%%%%%%%%%%%%%%%%%%%%%%%%%%%%%%
%%%%%%%%%%%%%%%%%%%%%%%%%%%%%%%%%%%%%%%%%%%%%%%%%%%

\section{The higher-order Fresnel pseudoprocess}\label{section2}
\noindent We start our analysis by studying the solution to the partial differential equation
\begin{equation}\label{horodseq_main}\frac{\partial^2 u}{\partial t^2}=(-1)^{n+1}\frac{\partial ^{2n}u}{\partial x^{2n}},\qquad x\in\mathbb{R},\;t>0,\;n\ge2.\end{equation}
\noindent By expressing equation (\ref{horodseq_main}) in the form
$$\left(\frac{\partial}{\partial t}-e^{-i\frac{\pi}{2}(n+1)}\frac{\partial^{n}}{\partial x^{n}}\right)\left(\frac{\partial}{\partial t}+e^{-i\frac{\pi}{2}(n+1)}\frac{\partial^{n}}{\partial x^{n}}\right)u=0$$
\noindent De Gregorio and Orsingher \cite{degregorioorsingher} showed that the fundamental solution $u_{2n}(x,t)$ to (\ref{horodseq_main}) can be expressed as a linear combination of the functions
\begin{equation}\label{partialsolutions}v_{1,n}(x,t)=\frac{1}{2\pi}\int_{-\infty}^{+\infty}e^{-i\gamma x-(-1)^nit\gamma^{n}}d\gamma,\qquad v_{2,n}(x,t)=\frac{1}{2\pi}\int_{-\infty}^{+\infty}e^{-i\gamma x+(-1)^nit\gamma^{n}}d\gamma.\end{equation}
\noindent We emphasize that a distinction must be made whether $n$ is even or odd.\\
\noindent For even $n$, say $n=2r$, the functions in (\ref{partialsolutions}) are complex-valued. A real-valued solution to equation (\ref{horodseq_main}) is obtained by considering an equally-weighted linear combination of $v_{1,n}(x,t)$ and $v_{2,n}(x,t)$. Thus, for $n=2r$, the solution to equation (\ref{horodseq_main}) can be expressed as
 \begin{align}
 u_{4r}(x,t)=&\frac{1}{2}\left[v_{1,2r}(x,t)+v_{2,2r}(x,t)\right]\nonumber\\
=&\frac{1}{4\pi}\left[\int_0^{+\infty}e^{-i\gamma x-it\gamma^{2r}}d\gamma+\int_{-\infty}^0e^{-i\gamma x-it\gamma^{2r}}d\gamma\right.\nonumber\\
\;&\;\;\;\left.+\int_0^{+\infty}e^{-i\gamma x+it\gamma^{2r}}d\gamma+\int_{-\infty}^0e^{-i\gamma x+it\gamma^{2r}}d\gamma\right]\nonumber\\
=&\frac{1}{2}\left[\int_0^{+\infty}\cos\left(\gamma x+t\gamma^{2r}\right)d\gamma + \int_0^{+\infty}\cos\left(-\gamma x+t\gamma^{2r}\right)d\gamma \right]\nonumber\\
=&\frac{1}{2(2rt)^{\frac{1}{2r}}}\left[\normalfont{\text{Ai}}_{2r}\left(\frac{x}{(2rt)^{\frac{1}{2r}}}\right)+\normalfont{\text{Ai}}_{2r}\left(-\frac{x}{(2rt)^{\frac{1}{2r}}}\right)\right]\label{evensol}
 \end{align}
 \noindent where $\normalfont{\text{Ai}}_{\alpha}(x)$ is the generalized Airy function defined in formula (\ref{fractionalairydef}).\\
 \noindent For odd $n$, say $n=2r+1$, the functions $v_{1,n}(x,t)$ and $v_{2,n}(x,t)$ are real-valued and can be expressed in terms of generalized Airy functions:
%$$v_{1,2r+1}(x,t)=\frac{1}{[(2r+1)t]^{\frac{1}{2r+1}}}\normalfont{\text{Ai}}_{2r+1}\left(-\frac{x}{[(2r+1)t]^{\frac{1}{2r+1}}}\right)$$
$$v_{j,2r+1}(x,t)=\frac{1}{[(2r+1)t]^{\frac{1}{2r+1}}}\normalfont{\text{Ai}}_{2r+1}\left((-1)^j\frac{x}{[(2r+1)t]^{\frac{1}{2r+1}}}\right),\qquad j=1,2.$$
\noindent Thus, any linear combination of the form
\begin{align}u_{2r+1}(x,t)=\frac{1}{[(2r+1)t]^{\frac{1}{2r+1}}}\Biggl[p\cdot &\normalfont{\text{Ai}}_{2r+1}\left(-\frac{x}{[(2r+1)t]^{\frac{1}{2r+1}}}\right)\nonumber\\
+q\cdot&\normalfont{\text{Ai}}_{2r+1}\left(\frac{x}{[(2r+1)t]^{\frac{1}{2r+1}}}\right)\Biggr],\qquad p,q\in\mathbb{R}\label{pq}\end{align}
\noindent is a solution to equation (\ref{horodseq_main}) for $n=2r+1$. For $p=q=\frac{1}{2}$, a symmetric solution analogous to (\ref{evensol}) is obtained.\\
\noindent For  $q=1-p$, $0\le p\le1$ the solution (\ref{pq}) admits the following probabilistic interpretation. Consider the pseudoprocess $X_{n}(t)$ having pseudo-density \begin{equation*}\mathbb{P}(X_{n}(t)\in dx)/dx=\frac{1}{(nt)^{\frac{1}{n}}}\normalfont{\text{Ai}}_{n}\left(\frac{x}{(nt)^{\frac{1}{n}}}\right).\end{equation*}
\noindent The pseudoprocess $X_{n}(t)$ emerges from the analysis of odd-order heat-type equations of the form (\ref{htype}) (see Orsingher \cite{orsingher} and Lachal \cite{lachal, lachal2}) and their space-fractional extensions  (Marchione and Orsingher \cite{marchione}). From formula (\ref{asymmetric_pseudodensity_frac_intro}) it is clear that
\begin{equation*}F_{2n}^{(p)}(t)\overset{i.d.}{=}
\begin{dcases}
-X_{n}(t)\qquad&\text{with prob.}\;p\\
X_{n}(t)\qquad&\text{with prob.}\;1-p.
\end{dcases}\end{equation*}
\noindent Thus, we interpret the Fresnel pseudoprocess as a pseudoprocess obtained by randomly choosing at the time $t=0$ the sign of $X_{n}(t)$.\\
\noindent In the final section of the paper we will briefly discuss the asymmetric solution (\ref{pq}) for $q=1-p$, $0\le p\le1$. Throughout this section we will restrict ourselves to the analysis of the symmetric case
\begin{equation}\label{generalsol_airy}u_{2n}(x,t)=\frac{1}{2(nt)^{\frac{1}{n}}}\left[\normalfont{\text{Ai}}_{n}\left(\frac{x}{(nt)^{\frac{1}{n}}}\right)+\normalfont{\text{Ai}}_{n}\left(-\frac{x}{(nt)^{\frac{1}{n}}}\right)\right]\end{equation}
\noindent for both $n$ even and odd. The fundamental solution (\ref{generalsol_airy}) arises from the following Cauchy problem.
\begin{theorem}The function (\ref{generalsol_airy}) is the solution to the Cauchy problem
\begin{equation*}\begin{dcases}\frac{\partial^2 u_{2n}}{\partial t^2}=(-1)^{n+1}\frac{\partial ^{2n}u_{2n}}{\partial x^{2n}}\\u_{2n}(x,0)=\delta(x),\;\;\frac{\partial u_{2n}}{\partial t}(x,0)=0.\end{dcases}\end{equation*}\end{theorem}
\begin{proof}The Fourier transform $\widehat{u}_{2n}(\gamma,t)=\int_{-\infty}^{+\infty}e^{i\gamma x}u_{2n}(x,t)d\gamma$ is the solution to the Cauchy problem
$$\begin{dcases}\frac{\partial^2 \widehat{u}_{2n}(\gamma,t)}{\partial t^2}=-\gamma^{2n}\widehat{u}_{2n}(\gamma,t)\\\widehat{u}_{2n}(\gamma,0)=1,\;\;\frac{\partial \widehat{u}_{2n}}{\partial t}(\gamma,0)=0.\end{dcases}$$
Thus, we obtain \begin{equation}\label{thmfou}\widehat{u}_{2n}(\gamma,t)=\cos(\gamma^nt)=\frac{1}{2}\left[e^{i\sgn(\gamma)|\gamma|^nt}+e^{-i\sgn(\gamma)|\gamma|^nt}\right].\end{equation}
\noindent By taking the inverse Fourier transform of formula (\ref{thmfou}) and having in mind the representation (\ref{fractionalairydef}) the proof is complete.
\end{proof}
 \noindent A power series representation for $u_{2n}(x,t)$ can be obtained by expressing the generalized Airy function in the following form (see Marchione and Orsingher \cite{marchione}):
 \begin{equation}\label{airypower}\normalfont{\text{Ai}}_{\alpha}(x)=\frac{1}{\pi\alpha^{\frac{\alpha-1}{\alpha}}}\sum_{k=0}^{\infty}\frac{x^k\alpha^{\frac{k}{\alpha}}}{k!}\Gamma\left(\frac{k+1}{\alpha}\right)\sin\left(\pi\frac{(k+1)(\alpha+1)}{2\alpha}\right),\qquad x\in\mathbb{R},\;\alpha>1.\end{equation}
 \noindent By combining the expressions (\ref{generalsol_airy}) and (\ref{airypower}), we have that
 \begin{equation}\label{seriessolution}u_{2n}(x,t)=\frac{1}{n\pi t^{\frac{1}{n}}}\sum_{k=0}^{\infty}\left(\frac{x}{t^{\frac{1}{n}}}\right)^{2k}\frac{\Gamma\left(\frac{2k+1}{n}\right)}{(2k)!}\sin\left(\pi\frac{(2k+1)(n+1)}{2n}\right)\end{equation}
 \noindent For $n=2$, by applying the gamma duplication formula $\frac{\Gamma\left(k+\frac{1}{2}\right)}{(2k)!}=\frac{\sqrt{\pi}}{2^k\Gamma(k+1)}$ and considering separately the odd-order terms and the even-order terms of the power series, formula (\ref{seriessolution}) becomes
\begin{align}
u_{4}(x,t)=&\frac{1}{2\sqrt{\pi t}}\sum_{k=0}^{\infty}\frac{1}{k!}\left(\frac{x^2}{4t}\right)^k\sin\left(\frac{3\pi}{4}(2k+1)\right)\nonumber\\
=&\frac{1}{2\sqrt{\pi t}}\left[\sum_{k=0}^{\infty}\frac{1}{(2k)!}\left(\frac{x^2}{4t}\right)^{2k}\sin\left(3k\pi+\frac{3}{4}\pi\right)\right.\nonumber\\
&\qquad+\left.\sum_{k=0}^{\infty}\frac{1}{(2k+1)!}\left(\frac{x^2}{4t}\right)^{2k+1}\sin\left(3k\pi+\frac{9}{4}\pi\right)\right]\nonumber\\
=&\frac{1}{2\sqrt{\pi t}}\left[\frac{1}{\sqrt{2}}\sum_{k=0}^{\infty}\frac{(-1)^k}{(2k)!}\left(\frac{x^2}{4t}\right)^{2k}+\frac{1}{\sqrt{2}}\sum_{k=0}^{\infty}\frac{(-1)^k}{(2k+1)!}\left(\frac{x^2}{4t}\right)^{2k+1}\right]\nonumber\\
=&\frac{1}{2\sqrt{\pi t}}\left[\frac{1}{\sqrt{2}}\cos\left(\frac{x^2}{4t}\right)+\frac{1}{\sqrt{2}}\sin\left(\frac{x^2}{4t}\right)\right]=\frac{1}{2\sqrt{\pi t}}\cos\left(\frac{x^2}{4t}-\frac{\pi}{4}\right)\label{cos}
\end{align}
which coincides with formula (\ref{fresneldensity}) up to a scaling factor.

\noindent In the following theorem we show that the power series representation (\ref{seriessolution}) can be interpreted as an expected value of amplified oscillations with Weibull distributed random parameters.

\begin{theorem}\label{weibullthm}The following relationship holds:
\begin{equation}u_{2n}(x,t)=\frac{1}{\pi x}\mathbb{E}\left[\sin\left(a_nxG_{n}(1/t)\right)\cosh\left(b_nxG_{n}(1/t)\right)\right]\end{equation}
\noindent where $G_{\gamma}(\tau)$ is a random variable with Weibull distribution having probability density function $$g_{\gamma}(y;\tau)=\gamma\frac{y^{\gamma-1}}{\tau}\exp\left(-\frac{y^{\gamma}}{\tau}\right),\qquad y,\gamma,\tau>0$$
and $$a_n=\cos\frac{\pi}{2n},\qquad b_n=\sin\frac{\pi}{2n}.$$
\end{theorem}
\begin{proof}
We start by writing formula (\ref{seriessolution}) in the form
\begin{align}u_{2n}(x,t)=&\frac{1}{\pi x}\sum_{k=0}^{\infty}\left(\frac{x}{t^{\frac{1}{n}}}\right)^{2k+1}\frac{\Gamma\left(1+\frac{2k+1}{n}\right)}{(2k+1)!}\sin\left(\pi\frac{(2k+1)(n+1)}{2n}\right)\nonumber\\
=&\frac{nt}{\pi x}\int_0^{+\infty}e^{-ty^{n}}y^{n-1}\sum_{k=0}^{\infty}\frac{(x y)^{2k+1}}{(2k+1)!}\sin\left(\pi\frac{(2k+1)(n+1)}{2n}\right)dy\label{thm_E1}
\end{align}
\noindent where we have used the integral representation of the gamma function. By using the relationship \begin{equation}\label{sinhseries}\sum_{k=0}^{\infty}\frac{x^{2k+1}}{(2k+1)!}\sin((2k+1)\phi)=\sin\left(x\sin\phi\right)\cosh(x\cos\phi)\end{equation}
\noindent formula (\ref{thm_E1}) becomes

\begin{equation*}u_{2n}(x,t)=\frac{nt}{\pi x}\int_0^{+\infty}e^{-ty^{n}}y^{n-1}\sin\left(xy\cos\frac{\pi}{2n}\right)\cosh\left(xy\sin\frac{\pi}{2n}\right)dy
\end{equation*}
\noindent which completes the proof.
\end{proof}

\noindent In view of formula (\ref{cos}), theorem \ref{weibullthm} implies that
\begin{equation}\label{expformula}\frac{1}{\pi x}\mathbb{E}\left[\sin\left(x\sqrt{\frac{Z}{2t}}\right)\cosh\left(x\sqrt{\frac{Z}{2t}}\right)\right]=\frac{1}{2\sqrt{\pi t}}\cos\left(\frac{x^2}{4t}-\frac{\pi}{4}\right)\end{equation}
\noindent where $Z$ is an exponential random variable with mean 1. By expressing the relationship (\ref{expformula}) in integral form we obtain the unexpected result
\begin{equation}\label{expformula_int}\frac{1}{\pi x}\int_0^{+\infty}\sin\left(x\sqrt{\frac{y}{2t}}\right)\cosh\left(x\sqrt{\frac{y}{2t}}\right)e^{-y}dy=\frac{1}{2\sqrt{\pi t}}\cos\left(\frac{x^2}{4t}-\frac{\pi}{4}\right).\end{equation}

\begin{remark}Formula (\ref{expformula_int}) can be obtained also by direct calculation. By observing that
$$\int_{-\infty}^{+\infty}we^{-w^2+Aw}dw=\frac{A\sqrt{\pi}}{2}\;e^{\frac{A^2}{4}},\qquad A\in\mathbb{C}.$$
\noindent we can write
\begin{align}&\frac{1}{\pi x}\int_0^{+\infty}\sin\left(x\sqrt{\frac{y}{2t}}\right)\cosh\left(x\sqrt{\frac{y}{2t}}\right)e^{-y}dy\nonumber\\
&=\frac{1}{8i\pi x}\int_{-\infty}^{+\infty}\left(e^{ix\sqrt{\frac{y}{2t}}}-e^{-ix\sqrt{\frac{y}{2t}}}\right)\left(e^{x\sqrt{\frac{y}{2t}}}+e^{-x\sqrt{\frac{y}{2t}}}\right)e^{-y}dy\nonumber\\
&=\frac{1}{4i\pi x}\int_{-\infty}^{+\infty}we^{-w^2}\left[e^{w\frac{x}{\sqrt{2t}}(1+i)}+e^{w\frac{x}{\sqrt{2t}}(i-1)}-e^{w\frac{x}{\sqrt{2t}}(1-i)}-e^{-w\frac{x}{\sqrt{2t}}(1+i)}\right]dw\nonumber\\
&=\frac{1}{4i\pi}\sqrt{\frac{\pi}{2t}}\left[(1+i)e^{i\frac{x^2}{4t}}+(i-1)e^{-i\frac{x^2}{4t}}\right]=\frac{1}{2\sqrt{2\pi t}}\left[\cos\left(\frac{x^2}{4t}\right)+\sin\left(\frac{x^2}{4t}\right)\right]\nonumber\\
&=\frac{1}{2\sqrt{\pi t}}\cos\left(\frac{x^2}{4t}-\frac{\pi}{4}\right).\nonumber\end{align}\end{remark}

\section{The fractional Fresnel pseudoprocess}
\noindent We now study a fractional generalization of the higher-order Fresnel pseudoprocess. In particular, we consider the pseudoprocess related to the space-fractional Cauchy problem
\begin{equation}\label{fractionalcau}\begin{dcases}\frac{\partial^2 u_{2\alpha}}{\partial t^2}(x,t)=\frac{\partial ^{2\alpha}u_{2\alpha}}{\partial |x|^{2\alpha}}(x,t),\qquad x\in\mathbb{R},\;t>0,\;\alpha>1\\u_{2\alpha}(x,0)=\delta(x),\;\;\frac{\partial u_{2\alpha}}{\partial t}(x,0)=0.\end{dcases}\end{equation}
\noindent where the operator $\frac{\partial ^{2\alpha}}{\partial |x|^{2\alpha}}$ denotes the Riesz fractional derivative. First proposed by Riesz \cite{riesz}, the Riesz fractional derivative can be implicitely defined by means of its Fourier transform
\begin{equation*}\mathcal{F}\left\{\frac{\partial ^{\alpha}f(x)}{\partial |x|^{\alpha}}\right\}(\gamma)=-|\gamma|^\alpha\mathcal{F}\{f(x)\}(\gamma).\end{equation*}

\noindent For a function $f :\mathbb{R}\to\mathbb{R}$, $f\in C_m(\mathbb{R})$, with derivatives decaying to 0 for $|x| \to +\infty$, the following explicit representation can be given to the Riesz fractional derivative of order $\alpha$, $m-1<\alpha< m$
$$\frac{\partial^{\alpha}f(x)}{\partial |x|^{\alpha}}=-\frac{1}{2\Gamma(m-\alpha)\cos\left(\frac{\pi\alpha}{2}\right)}\frac{d^m}{dx^m}\left[\int_{-\infty}^x\frac{f(z)}{(x-z)^{\alpha-m+1}}dz+(-1)^m\int_{x}^{\infty}\frac{f(z)}{(z-x)^{\alpha-m+1}}dz\right].$$

\noindent In the next theorem we show that the Riesz fractional differential operator is related to a fractional extension of the higher-order Fresnel pseudoprocess examined in the previous section.
\begin{theorem}The solution to the fractional Cauchy problem (\ref{fractionalcau}) can be expressed in the form

\begin{equation}\label{frac_generalsol_airy}u_{2\alpha}(x,t)=\frac{1}{2(\alpha t)^{\frac{1}{\alpha}}}\left[\normalfont{\text{Ai}}_{\alpha}\left(\frac{x}{(\alpha t)^{\frac{1}{\alpha}}}\right)+\normalfont{\text{Ai}}_{\alpha}\left(-\frac{x}{(\alpha t)^{\frac{1}{\alpha}}}\right)\right],\qquad x\in\mathbb{R},\;t>0,\;\alpha>1.\end{equation}
\end{theorem}
\begin{proof}The Fourier transform $\widehat{u}_{2\alpha}(\gamma,t)=\int_{-\infty}^{+\infty}e^{i\gamma x}u_{2\alpha}(x,t)d\gamma$ is the solution to the Cauchy problem
$$\begin{dcases}\frac{\partial^2 \widehat{u}_{2\alpha}(\gamma,t)}{\partial t^2}=-|\gamma|^{2\alpha}\widehat{u}_{2\alpha}(\gamma,t)\\\widehat{u}_{2\alpha}(\gamma,0)=1,\;\;\frac{\partial \widehat{u}_{2\alpha}}{\partial t}(\gamma,0)=0.\end{dcases}$$
Thus, we obtain \begin{equation}\label{frac_thmfou}\widehat{u}_{2\alpha}(\gamma,t)=\cos(|\gamma|^\alpha t)=\frac{1}{2}\left[e^{i\sgn(\gamma)|\gamma|^\alpha t}+e^{-i\sgn(\gamma)|\gamma|^\alpha t}\right].\end{equation}
\noindent By taking the inverse Fourier transform of (\ref{frac_thmfou}) and considering formula (\ref{fractionalairydef}) the proof is complete.
\end{proof}

\noindent By means of the usual construction described in (\ref{signedmeasuredef}), the function $u_{2\alpha(x,t)}$ can be viewed as the density of a signed measure on the space of real-valued measurable functions. Similarly to the integer case, the pseudo-density admits a power series representation as
\begin{equation}\label{frac_seriessolution}u_{2\alpha}(x,t)=\frac{1}{\alpha\pi t^{\frac{1}{\alpha}}}\sum_{k=0}^{\infty}\left(\frac{x}{t^{\frac{1}{\alpha}}}\right)^{2k}\frac{\Gamma\left(\frac{2k+1}{\alpha}\right)}{(2k)!}\sin\left(\pi\frac{(2k+1)(\alpha+1)}{2\alpha}\right),\qquad\alpha>1\end{equation}
\noindent which can be easily obtained by using (\ref{airypower}). Moreover, the pseudo-density (\ref{frac_seriessolution}) admits a probabilistic representation in terms of an expected value of amplified oscillations with Weibull distributed random parameters
\begin{equation}\label{frac_cosh}u_{2\alpha}(x,t)=\frac{1}{\pi x}\mathbb{E}\left[\sin\left(a_\alpha xG_{\alpha}(1/t)\right)\cosh\left(b_\alpha xG_{\alpha}(1/t)\right)\right]\end{equation}
\noindent where $G_{\gamma}(\tau)$ is a random variable with Weibull distribution as in theorem \ref{weibullthm} and $a_\alpha=\cos\frac{\pi}{2\alpha}$, $b_\alpha=\sin\frac{\pi}{2\alpha}.$
\noindent We omit the proof of formula (\ref{frac_cosh}) which coincides with the proof of theorem \ref{weibullthm} in view of (\ref{frac_seriessolution}).
%\begin{theorem}\label{frac_weibullthm}The following relationship holds:
%\begin{equation}\label{frac_cosh}u_{2\alpha}(x,t)=\frac{1}{\pi x}\mathbb{E}\left[\sin\left(a_\alpha xG_{\alpha}(1/t)\right)\cosh\left(b_\alpha xG_{\alpha}(1/t)\right)\right]\end{equation}
%\noindent where $G_{\gamma}(\tau)$ is a random variable with Weibull distribution having probability density function $$g_{\gamma}(y;\tau)=\gamma\frac{y^{\gamma-1}}{\tau}\exp\left(-\frac{y^{\gamma}}{\tau}\right),\qquad y,\gamma,\tau>0$$
%and $$a_\alpha=\cos\frac{\pi}{2\alpha},\qquad b_\alpha=\sin\frac{\pi}{2\alpha}.$$
%\end{theorem}
%\begin{proof}
%\noindent By expressing formula (\ref{frac_seriessolution}) in the form
%\begin{align}u_{2\alpha}(x,t)=&\frac{1}{\pi x}\sum_{k=0}^{\infty}\left(\frac{x}{t^{\frac{1}{\alpha}}}\right)^{2k+1}\frac{\Gamma\left(1+\frac{2k+1}{\alpha}\right)}{(2k+1)!}\sin\left(\pi\frac{(2k+1)(\alpha+1)}{2\alpha}\right)\nonumber\\
%=&\frac{\alpha t}{\pi x}\int_0^{+\infty}e^{-ty^{\alpha}}y^{\alpha-1}\sum_{k=0}^{\infty}\frac{(x y)^{2k+1}}{(2k+1)!}\sin\left(\pi\frac{(2k+1)(\alpha+1)}{2\alpha}\right)dy\nonumber
%\end{align}
%\noindent formula (\ref{sinhseries}) permits us to write
%\begin{equation*}u_{2\alpha}(x,t)=\frac{\alpha t}{\pi x}\int_0^{+\infty}e^{-ty^{\alpha}}y^{\alpha-1}\sin\left(xy\cos\frac{\pi}{2\alpha}\right)\cosh\left(xy\sin\frac{\pi}{2\alpha}\right)dy
%\end{equation*}
%\noindent which completes the proof.
%\end{proof}

%%%%%%%%%%%%%%%%%%%%%%%%%%%%%%%%%%%%%%%%%%%%%%%%%%%%%%%%%%%%%
%%%%%%%%%%%%%%%%%%%%%%%%%%%%%%%%%%%%%%%%%%%%%%%%%%%%%%%%%%%%%
%%%%%%%%%%%%%%%%%%%%%%%%%%%%%%%%%%%%%%%%%%%%%%%%%%%%%%%%%%%%%
\section{Subordinated Fresnel pseudoprocesses}
\noindent In this section we study the fractional Fresnel pseudoprocess $F_{2\alpha}(t)$ time-changed with an independent stable subordinator $S_{\theta}(t)$ having characteristic function
\begin{equation}\label{subchf}\mathbb{E}\left[e^{i\gamma S_{\theta}(t)}\right]=e^{-t|\gamma|^{\theta}e^{-i\frac{\pi\theta}{2}\sgn\gamma}}.\end{equation} Our first result concerns the pseudo-density of the subordinated pseudoprocess, for which a power series representation is obtained in the following theorem.
\begin{theorem}Let $F_{2\alpha}(t)$ be the Fresnel pseudoprocess of order $\alpha>1$ and $S_{\theta}(t)$ an independent stable subordinator of exponent $\theta$ with characteristic function (\ref{subchf}). If $\alpha\theta>1$ the following relationship holds:
\begin{align}&\mathbb{P}\left(F_{2\alpha}(S_{\theta}(t))\in dx\right)/dx\nonumber\\
\;\;&=\frac{1}{\alpha\theta\pi t^{\frac{1}{\alpha\theta}}}\sum_{k=0}^{\infty}\left(\frac{x}{t^{\frac{1}{\alpha\theta}}}\right)^{2k}\frac{\Gamma\left(\frac{2k+1}{\alpha\theta}\right)}{(2k)!}\sin\left(\pi\frac{(2k+1)(\alpha+1)}{2\alpha}\right),\;\; x\in\mathbb{R}.\label{finalsectionthm1thesis}\end{align}
\end{theorem}
\begin{proof}
We use the Wright function representation of the probability density function $h_{\theta}(x,t)$ of the stable subordinator $S_{\theta}(t)$
\begin{equation*}h_{\theta}(x,t)=\frac{\theta t}{x^{\theta+1}}W_{-\theta,1-\theta}\left(-\frac{t}{x^{\theta}}\right)=\frac{\theta}{x}\sum_{k=0}^{\infty}\frac{(-1)^kt^{k+1}}{x^{\theta(k+1)}k!\;\Gamma(-\theta(k+1)+1)},\qquad x>0,\;t>0.\end{equation*}
In light of (\ref{frac_seriessolution}) we can write
\begin{align}\mathbb{P}( F_{2\alpha}&(S_{\theta}(t))\in dx)/dx=\int_0^{+\infty}u_{2\alpha}(x,s)h_{\theta}(s,t)ds\nonumber\\
=&\frac{\theta t}{\alpha\pi}\sum_{k=0}^{\infty}\frac{x^{2k}}{(2k)!}\;\Gamma\left(\frac{2k+1}{\alpha}\right)\sin\left(\pi\frac{(2k+1)(\alpha+1)}{2\alpha}\right)\cdot\nonumber\\
&\hspace{5cm}\cdot\int_0^{+\infty}s^{-\theta-\frac{2k+1}{\alpha}-1}W_{-\theta,1-\theta}\left(-\frac{t}{s^{\theta}}\right)ds\nonumber\\
=&\frac{1}{\alpha\pi t^{\frac{1}{\alpha\theta}}}\sum_{k=0}^{\infty}\left(\frac{x}{t^{\frac{1}{\alpha\theta}}}\right)^{2k}\frac{\Gamma\left(\frac{2k+1}{\alpha}\right)}{(2k)!}\sin\left(\pi\frac{(2k+1)(\alpha+1)}{2\alpha}\right)\cdot\nonumber\\
&\hspace{5cm}\cdot\int_0^{+\infty}y^{\frac{2k+1}{\alpha\theta}}W_{-\theta,1-\theta}\left(-y\right)dy.\label{finalsec_thm1}\end{align}
\noindent By using the Mellin transform of the Wright function (see Prudnikov et al. \cite{prudnikov}, pag. 355) 
$$\int_0^{+\infty}\text{W}_{a,b}(-x)\;x^{\eta-1}dx=\frac{\Gamma(\eta)}{\Gamma(b-a\eta)},\qquad\eta>0,\;a>-1,\;b\in\mathbb{C}$$
\noindent the representation (\ref{finalsectionthm1thesis}) is easily obtained. The convergence of the power series for $\alpha\theta>1$ can be proved by using the Stirling formula for the gamma function.
%\noindent formula (\ref{finalsec_thm1}) becomes
%$$\mathbb{P}( F_{2\alpha}(S_{\theta}(t))\in dx)/dx=\frac{1}{\alpha\theta\pi t^{\frac{1}{\alpha\theta}}}\sum_{k=0}^{\infty}\left(\frac{x}{t^{\frac{1}{\alpha\theta}}}\right)^{2k}\frac{\Gamma\left(\frac{2k+1}{\alpha\theta}\right)}{(2k)!}\sin\left(\pi\frac{(2k+1)(n+1)}{2\alpha}\right).$$
\end{proof}
\noindent The next theorem shows that the probabilistic representation in terms of amplified oscillations with Weibull distributed random parameters is maintained if the Fresnel pseudoprocess is time-changed with an independent stable subordinator.

\begin{theorem}\label{frac_weibullthm}For $\alpha\theta>1$ the following relationship holds:
\begin{equation}\mathbb{P}\left(F_{2\alpha}(S_{\theta}(t))\in dx\right)/dx=\frac{1}{\pi x}\mathbb{E}\left[\sin\left(a_{\alpha}xG_{\alpha\theta}(1/t)\right)\cosh\left(b_{\alpha}xG_{\alpha\theta}(1/t)\right)\right]\end{equation}
\noindent where $G_{\gamma}(\tau)$ is a random variable with Weibull distribution having probability density function $$g_{\gamma}(y;\tau)=\gamma\frac{y^{\gamma-1}}{\tau}\exp\left(-\frac{y^{\gamma}}{\tau}\right),\qquad y,\gamma,\tau>0$$
and $a_{\alpha}=\cos\frac{\pi}{2\alpha},\qquad b_{\alpha}=\sin\frac{\pi}{2\alpha}.$
\end{theorem}
\begin{proof}
By using the integral representation of the gamma function in formula (\ref{finalsectionthm1thesis}) we can write
\begin{align}\mathbb{P}(F_{2\alpha}(S_{\theta}(t))\in dx)/dx=&\frac{1}{\pi x}\sum_{k=0}^{\infty}\left(\frac{x}{t^{\frac{1}{\alpha\theta}}}\right)^{2k+1}\frac{\Gamma\left(\frac{2k+1}{\alpha\theta}+1\right)}{(2k+1)!}\sin\left(\pi\frac{(2k+1)(n+1)}{2\alpha}\right)\nonumber\\
=&\frac{\alpha\theta t}{\pi x}\int_0^{+\infty}e^{-ty^{\alpha\theta}}y^{\alpha\theta-1}\sum_{k=0}^{\infty}\frac{(yx)^{2k+1}}{(2k+1)!}\sin\left(\pi\frac{(2k+1)(\alpha+1)}{2\alpha}\right)dy\nonumber\end{align}
By using formula (\ref{sinhseries}) as in theorem \ref{weibullthm}, we obtain
\begin{equation*}\mathbb{P}(F_{2\alpha}(S_{\theta}(t))\in dx)/dx=\frac{\alpha\theta t}{\pi x}\int_0^{+\infty}e^{-ty^{\alpha\theta}}y^{\alpha\theta-1}\sin\left(xy\cos\frac{\pi}{2\alpha}\right)\cosh\left(xy\sin\frac{\pi}{2\alpha}\right)dy.\end{equation*}
\end{proof}

\noindent We now show that the distribution of the time-changed Fresnel pseudoprocess $$Y_{2\alpha,\theta}(t)=F_{2\alpha}(S_{\theta}(t))$$
\noindent is a genuine probability distribution if the exponent of the stable subordinator $S_{\theta}(t)$ is chosen in a suitable way. We start by observing that the pseudoprocess $Y_{2\alpha,\theta}(t)$  has characteristic function
\begin{align}\mathbb{E}&\left[e^{i\gamma Y_{2\alpha,\theta}(t)}\right]=\mathbb{E}\left[\mathbb{E}\left[e^{i\gamma F_{2\alpha}(S_{\theta}(t))}\big\lvert S_{\theta}(t)\right]\right]=\mathbb{E}\left[\cos\left(\gamma^{\alpha} S_{\theta}(t)\right)\right]\nonumber\\
&\;\;\;=\frac{1}{2}\mathbb{E}\left[e^{i\sgn(\gamma)|\gamma|^{\alpha}S_{\theta}(t)}+e^{-i\sgn(\gamma)|\gamma|^{\alpha}S_{\theta}(t)}\right]\nonumber =\frac{1}{2}\left[e^{-t|\gamma|^{\alpha\theta}e^{i\frac{\theta\pi}{2}\sgn\gamma}}+e^{-t|\gamma|^{\alpha\theta}e^{-i\frac{\theta\pi}{2}\sgn\gamma}}\right]\nonumber \\
&\;\;\;=\frac{1}{2}\left[e^{-t|\gamma|^{\alpha\theta}\cos\frac{\pi\theta}{2}\left(1+i\tan\frac{\pi\theta}{2}\sgn\gamma\right)}+e^{-t|\gamma|^{\alpha\theta}\cos\frac{\pi\theta}{2}\left(1-i\tan\frac{\pi\theta}{2}\sgn\gamma\right)}\right]\label{genuineprob}
\end{align}  
\noindent where we have used formulas (\ref{thmfou}) and (\ref{subchf}).\\
\noindent The characteristic function of a stable process $H_{\nu}(\sigma, \beta, \mu; t)$ of exponent $\nu\neq1$ can be expressed in the form
\begin{equation}\label{stableprocessescharfun}\mathbb{E}\left[e^{i\gamma H_{\nu}(\sigma, \beta, \mu; t)}\right]=e^{-t\sigma^{\nu}|\gamma|^{\nu}\left(1-i\beta\sgn\gamma\tan\frac{\pi\nu}{2}\right)+i\mu t\gamma}\end{equation}
\noindent where $0<\nu\le2,\;\nu\neq1$ is the exponent of the stable process, $\sigma>0$ is the dispersion parameter, $\beta\in[-1,1]$ is the skewness parameter and $\mu\in\mathbb{R}$ is the drift parameter.\\
\noindent Denote by $Z_{\nu}(\sigma, \beta, \mu; t)$ the random variable
\begin{equation}\label{mix}Z_{\nu}(\sigma, \beta, \mu; t)=
\begin{dcases}
H_{\nu}(\sigma, \beta, \mu; t)\qquad&\text{with prob.}\;\frac{1}{2}\\
-H_{\nu}(\sigma, \beta, \mu; t)\qquad&\text{with prob.}\;\frac{1}{2}
\end{dcases}\end{equation}
\noindent where where $H_{\nu}(\sigma, \beta, \mu; t)$ is a stable random variable with characteristic function  (\ref{stableprocessescharfun}).
\noindent By setting
\begin{equation}\label{paramvalues}\nu=\alpha\theta,\qquad\beta=-\frac{\tan\frac{\pi\theta}{2}}{\tan\frac{\pi \alpha\theta}{2}},\qquad\sigma=\left(\cos\frac{\pi\theta}{2}\right)^{\frac{1}{\nu}},\qquad\mu=0\end{equation}
\noindent formula (\ref{genuineprob}) implies that 
\begin{equation}\label{indistribution}Y_{2\alpha,\theta}(t)\overset{i.d.}{=}Z_{\nu}(\sigma, \beta, \mu; t)\end{equation}
\noindent Observe that the result (\ref{indistribution}) holds if the constraints $\nu\in(0,2]$ and $\beta\in[-1,1]$ are satisfied. As discussed by Marchione and Orsingher \cite{marchione}, for fixed $\nu\in(0,2)$ and $\beta\in(-1,1),\;\beta\neq 0,$ it is always possible to choose $\alpha>1$ and $\theta\in(0,1)$ such that the relationships (\ref{paramvalues}) are satisfied.\\
\noindent Formula (\ref{mix}) shows that the Fresnel pseudoprocess time-changed with a stable subordinator is identical in distribution, for all $t>0$, to a stochastic process which is obtained by randomly choosing at time $t=0$ the sign of an asymmetric stable process. Since the positive and negative signs are equiprobable, the resulting random process is symmetric. As stable distributions are unimodal, one may expect the random variable (\ref{indistribution}) to have a bimodal distribution. We will now show that this is not necessarely true.\\
\noindent Consider the particular case of a Fresnel pseudoprocess $F_{2\alpha}(t)$ of order $2\alpha,\;\alpha>1$ time-changed with a stable subordinator $S_{\frac{1}{\alpha}}(t)$. The subordinated pseudoprocess
$$Y_{2\alpha,\frac{1}{\alpha}}(t)=F_{2\alpha}(S_{\frac{1}{\alpha}}(t))$$
\noindent has characteristic function
\begin{equation}\label{doublecauchycf}\mathbb{E}\left[e^{i\gamma Y_{2\alpha,\frac{1}{\alpha}}(t)}\right]=\frac{1}{2}\left[e^{-t|\gamma|\cos\frac{\pi}{2\alpha}\left(1+i\tan\frac{\pi}{2\alpha}\sgn\gamma\right)}+e^{-t|\gamma|\cos\frac{\pi}{2\alpha}\left(1-i\tan\frac{\pi}{2\alpha}\sgn\gamma\right)}\right]\end{equation}
\noindent whose inverse Fourier transform coincides with the mixture of two asymmetric Cauchy probability density functions. The resulting distribution is symmetric and reads
\begin{align}\mathbb{P}\left(Y_{2\alpha,\frac{1}{\alpha}}(t)\in dx\right)/dx=&\frac{1}{2\pi}\left[\frac{t\cos\frac{\pi}{2\alpha}}{x^2+t^2-2xt\sin\frac{\pi}{2\alpha}}+\frac{t\cos\frac{\pi}{2\alpha}}{x^2+t^2+2xt\sin\frac{\pi}{2\alpha}}\right]\nonumber\\
=&\frac{t\cos\frac{\pi}{2\alpha}}{\pi}\left[\frac{x^2+t^2}{x^4+t^4+2x^2t^2\cos\frac{\pi}{\alpha}}\right].\label{doublecauchydensity}\end{align}
\noindent The two modal points of the probability density function (\ref{doublecauchydensity}) are located at $x=\pm t\sin\frac{\pi}{2\alpha}$. For $\alpha=2$, our result coincides with formula (5.4) of the paper by Orsingher and D'Ovidio \cite{orsingherdovidio_fresnel}. For $1<\alpha<3$, the first derivative of the probability density function (\ref{doublecauchydensity}) has three roots in $x=0$ and $x=\pm t\sqrt{2\sin\frac{\pi}{2\alpha}-1}$. Thus, the density function has two modes in $x=\pm t\sqrt{2\sin\frac{\pi}{2\alpha}-1}$ and a minimum at the origin. However, for $\alpha\ge 3$, the only real root of the first derivative is in $x=0$. The mixture density is unimodal and reaches its maximum at the origin.

%%%%%%%%%%%%%%%%%%%%%%%%%%%%%%%%%%%%%%%%%%%%%%%%%%%
%%%%%%%%%%%%%%%%%%%%%%%%%%%%%%%%%%%%%%%%%%%%%%%%%%%
%%%%%%%%%%%%%%%%%%%%%%%%%%%%%%%%%%%%%%%%%%%%%%%%%%%
%%%%%%%%%%%%%%%%%%%%%%%%%%%%%%%%%%%%%%%%%%%%%%%%%%%

\section{The asymmetric Fresnel pseudoprocess}
\noindent In section \ref{section2} we have introduced asymmetric solutions to equation (\ref{horodseq_main}) for odd values of $n$. A natural fractional extension of such solutions is obtained by considering the function
\begin{equation}\label{asymmetric_pseudodensity_frac}u_{2\alpha, p}(x,t)=\frac{1}{(\alpha t)^{\frac{1}{\alpha}}}\Biggl[p\cdot \normalfont{\text{Ai}}_{\alpha}\left(-\frac{x}{(\alpha t)^{\frac{1}{\alpha}}}\right)+(1-p)\cdot\normalfont{\text{Ai}}_{\alpha}\left(\frac{x}{(\alpha t)^{\frac{1}{\alpha}}}\right)\Biggr],\;\; x\in\mathbb{R},\;t>0\end{equation}
\noindent where $p\in[0,1]$ and $\alpha>1$ as usual. It can be checked that the Fourier transform of the pseudo-density (\ref{asymmetric_pseudodensity_frac}) satisfies the ordinary differential equation
$$\frac{\partial^2 \widehat{u}_{2\alpha,p}(\gamma,t)}{\partial t^2}=-|\gamma|^{2\alpha}\widehat{u}_{2\alpha,p}(\gamma,t)$$
\noindent which implies that, similarly to the symmetric case, the pseudo-density is a solution to the fractional partial differential equation
$$\frac{\partial^2 u}{\partial t^2}=\frac{\partial ^{2\alpha}u}{\partial |x|^{2\alpha}},\qquad x\in\mathbb{R},\;t>0.$$
\noindent Moreover, by using the following probabilistic representation for the Airy function (Marchione and Orsingher \cite{marchione})
\begin{equation}\label{probabilistic_airy}\frac{1}{(\alpha t)^{\frac{1}{\alpha}}}\normalfont{\text{Ai}}_{\alpha}\left(\frac{x}{(\alpha t)^{\frac{1}{\alpha}}}\right)=\frac{1}{\pi x}\mathbb{E}\left[\sin\left(a_\alpha xG_{\alpha}(1/t)\right)e^{-b_\alpha xG_{\alpha}(1/t)}\right]\end{equation}
we have that 
\begin{equation}u_{2\alpha,p}(x,t)=\frac{1}{\pi x}\mathbb{E}\left[\sin\left(a_\alpha xG_{\alpha}(1/t)\right)\left(pe^{b_\alpha xG_{\alpha}(1/t)}+(1-p)e^{-b_\alpha xG_{\alpha}(1/t)}\right)\right]\label{weighted_exps}\end{equation}
\noindent where we have used the same notation of theorem \ref{frac_weibullthm}. Of course, for $p=\frac{1}{2}$, formula (\ref{weighted_exps}) coincides with (\ref{frac_cosh}).\\
\noindent Denote by $F_{2\alpha}^{(p)}(t)$ the asymmetric Fresnel pseudoprocess having pseudo-density (\ref{asymmetric_pseudodensity_frac}) and let $S_{\theta}(t)$ be a stable subordinator with characteristic function (\ref{subchf}) such that $F_{2\alpha}^{(p)}(t)$ and $S_{\theta}(t)$ are independent. Consider the subordinated pseudoprocess $$Y_{2\alpha,\theta}^{(p)}(t)=F_{2\alpha}^{(p)}(S_{\theta}(t)).$$
\noindent It can be verified that $Y_{2\alpha,\theta}^{(p)}(t)$ has characteristic function
\begin{align}\mathbb{E}\left[e^{i\gamma Y_{2\alpha,\theta}^{(p)}(t)}\right]=\left[p\cdot e^{-t|\gamma|^{\alpha\theta}\cos\frac{\pi\theta}{2}\left(1+i\tan\frac{\pi\theta}{2}\sgn\gamma\right)}+(1-p)\cdot e^{-t|\gamma|^{\alpha\theta}\cos\frac{\pi\theta}{2}\left(1-i\tan\frac{\pi\theta}{2}\sgn\gamma\right)}\right].\label{asymm_genuineprob}
\end{align}  

\noindent Denote by $Z_{\nu}(\sigma, \beta, \mu; t)$ the random variable
\begin{equation}\label{asymm_mix}Z_{\nu}(\sigma, \beta, \mu; t)=
\begin{dcases}
H_{\nu}(\sigma, \beta, \mu; t)\qquad&\text{with prob.}\;p\\
-H_{\nu}(\sigma, \beta, \mu; t)\qquad&\text{with prob.}\;1-p
\end{dcases}\end{equation}
\noindent where $H_{\nu}(\sigma, \beta, \mu; t)$ is a stable random variable with characteristic function  (\ref{stableprocessescharfun}). Formula (\ref{asymm_genuineprob}) implies that
\begin{equation}\label{asymm_indistribution}Y_{2\alpha,\theta}^{(p)}(t)\overset{i.d.}{=}Z_{\nu}(\sigma, \beta, \mu; t)\end{equation}
\noindent where we have used the parametrization (\ref{paramvalues}).
\noindent Thus, the asymmetric Fresnel pseudoprocess time-changed with a suitable stable subordinator is identical in distribution, for all $t>0$, to a stochastic process obtained by randomly choosing at the initial time $t=0$ the sign of an asymmetric stable process. The resulting random process is asymmetric for $p\neq\frac{1}{2}$.\\
\noindent By considering the special case of a mixture of Cauchy distributions, we have shown in the previous section that the symmetric mixture density of the process (\ref{mix}) can have one or three stationary points. In the first case the distribution is unimodal while in the second case it is bimodal. We will now discuss a third possibile case which can occur for the asymmetric process (\ref{asymm_mix}). For suitable choices of the parameters, the probability density function can have exactly two stationary points, one being a maximum point and the other one being an inflection point. Consider now an asymmetric mixture of Cauchy distributions, which emerges from the subordinated pseudoprocess
$$Y_{2\alpha,\frac{1}{\alpha}}^{(p)}(t)=F_{2\alpha}^{(p)}(S_{\frac{1}{\alpha}}(t)).$$
\noindent By using formula (\ref{asymm_genuineprob}) for $\theta=\frac{1}{\alpha},\;\alpha>1$
\noindent it can be shown that
\begin{align}\mathbb{P}\left(Y_{2\alpha,\frac{1}{\alpha}}^{(p)}(t)\in dx\right)/dx=&\frac{t\cos\frac{\pi}{2\alpha}}{\pi}\left[\frac{p}{x^2+t^2-2xt\sin\frac{\pi}{2\alpha}}+\frac{1-p}{x^2+t^2+2xt\sin\frac{\pi}{2\alpha}}\right]\nonumber\\
=&\frac{t\cos\frac{\pi}{2\alpha}}{\pi}\left[\frac{x^2+t^2+2(2p-1)xt\sin\frac{\pi}{2\alpha}}{x^4+t^4+2x^2t^2\cos\frac{\pi}{\alpha}}\right].\label{asymm_doublecauchydensity}\end{align}
\noindent In order to study the stationary points of the density (\ref{asymm_doublecauchydensity}), we should study in principle its first and second derivatives. By setting $f(x)=\mathbb{P}\left(Y_{2\alpha,\frac{1}{\alpha}}^{(p)}(t)\in dx\right)/dx$, we have that
\begin{align}f'(x)=\frac{2t\cos\frac{\pi}{2\alpha}}{\pi}\bigg[&\frac{\left(x+(2p-1)t\sin\frac{\pi}{2\alpha}\right)\left(x^4+t^4+2x^2t^2\cos\frac{\pi}{\alpha}\right)}{\left(x^4+t^4+2x^2t^2\cos\frac{\pi}{\alpha}\right)^2}\nonumber\\
&-\frac{2x\left(x^2+t^2+2xt(2p-1)\sin\frac{\pi}{2\alpha}\right)(x^2+t^2\cos\frac{\pi}{\alpha})}{\left(x^4+t^4+2x^2t^2\cos\frac{\pi}{\alpha}\right)^2}\bigg].\label{fifth}\end{align}
\noindent The numerator of the expression (\ref{fifth}) is a fifth-order polynomial in $x$. Its roots are difficult to find out and it is not even guaranteed that they are expressible in terms of radicals. The study of the fifth-order algebraic equation emerging from formula (\ref{fifth}) is beyond the aim of our paper. However, in order to prove that the density (\ref{asymm_doublecauchydensity}) can have an inflection point, it is sufficient for us to find a specific combination of $\alpha$ and $p$ for which the inflection point occurs.\\
\noindent A substantial simplification of formula (\ref{fifth}) can be obtained if $\alpha$ and $p$ are chosen in such a way that the terms $x+(2p-1)t\sin\frac{\pi}{2\alpha}$ and $x^2+t^2\cos\frac{\pi}{\alpha}$ in the numerator have the same root. Thus, by setting
\begin{equation}\label{rootrel}-(2p-1)t\sin\frac{\pi}{2\alpha}=\pm t\sqrt{-\cos\frac{\pi}{\alpha}}\end{equation}
\noindent we immediately obtain \begin{equation}p=\frac{\sin\frac{\pi}{2\alpha}\mp\sqrt{-\cos\frac{\pi}{\alpha}}}{2\sin\frac{\pi}{2\alpha}}.\label{p(a)}\end{equation}
\noindent This implies that, if $p$ satisfies the relationship (\ref{p(a)}), the point
\begin{equation}x=\pm t\sqrt{-\cos\frac{\pi}{\alpha}}\label{root}\end{equation}
is a root of (\ref{fifth}) and thus it is a stationary point of (\ref{asymm_doublecauchydensity}). Observe that, in order for (\ref{root}) to be a real root, we need to impose the constraint $1<\alpha<2$. Moreover, since $\sin\frac{\pi}{2\alpha}>\sqrt{-\cos\frac{\pi}{\alpha}}$ for $1<\alpha<2$, the constraint $p\in[0,1]$ is satisfied by formula (\ref{p(a)}). Finally, note that the choice of the sign in formula (\ref{root}) is arbitrary as long as the correct sign is then chosen in formula (\ref{p(a)}).\\
\noindent Our next step is to determine the value of $\alpha$ such that the stationary point (\ref{root}) is an inflection point. By taking into account formula (\ref{rootrel}), tedious but straightforward calculation permits to prove that
\begin{equation}\label{f22}f''\left(\pm t\sqrt{-\cos\frac{\pi}{\alpha}}\right)=\frac{1}{\pi}\frac{3\cos^2\frac{\pi}{2\alpha}-1}{2t^3\cos\frac{\pi}{2\alpha}\sin^4\frac{\pi}{2\alpha}}.\end{equation}
\noindent Of course, the second derivative (\ref{f22}) vanishes for
\begin{equation}\label{alpha*}\alpha=\frac{\pi}{2\arccos\left(\frac{1}{\sqrt{3}}\right)}.\end{equation}
\noindent Moreover, since $\frac{\pi}{4}<\arccos\left(\frac{1}{\sqrt{3}}\right)<\frac{\pi}{2}$, the constraint $1<\alpha<2$ is satisfied by formula (\ref{alpha*}). Thus, by setting $p=\frac{\sqrt{2}\mp1}{2\sqrt{2}}$ as in formula (\ref{p(a)}), we have found a particular combination of $\alpha$ and $p$ for which the probability density function (\ref{asymm_doublecauchydensity}) has an inflection point in $$x=\pm\frac{t}{\sqrt{3}}.$$
\noindent For such a choice of $\alpha$ and $p$, we can rewrite formula (\ref{fifth}) as
\begin{equation}f'(x)=-\frac{2t}{\pi\sqrt{3}}\;\frac{\left(x\mp\frac{t}{\sqrt{3}}\right)^2\left(x^3\mp \frac{t}{\sqrt{3}} x^2 + t^2x \pm\sqrt{3}t^3\right)}{\left(x^4+t^4-\frac{2t^2}{3}x^2\right)^2}\label{final}\end{equation}
\noindent The expression (\ref{final}) has two coinciding real roots in $x=\pm\frac{t}{\sqrt{3}}$, which represent the inflection point of the density, and a real root which represents the maximum point. Observe that the expression $x^3\mp \frac{t}{\sqrt{3}} x^2 + t^2x \pm\sqrt{3}t^3$ in the numerator of (\ref{final}) is strictly monotonically increasing and thus it has only one real root.

\end{document}